%% file: SiegelDisksRev1.v7.tex
\documentclass[11pt, twoside]{article}
\usepackage{amsfonts,amssymb,amsmath,amsthm}
\usepackage{graphicx,url}
\usepackage[all]{xy}

\usepackage{color,psfrag,xmpmulti,amscd, import}
\newtheorem*{mainthm}{Main Theorem (Simplified form)}

\newtheorem*{propA}{Proposition A}
\newtheorem*{propB}{Proposition B}
\newtheorem*{propC}{Proposition C}
\newtheorem*{mainthm2}{Main Theorem (General form)}

\newtheorem*{prop1}{Proposition 1}

\newtheorem*{theorem}{Theorem}
\newtheorem*{corollary}{Corollary}

\theoremstyle{remark}
\newtheorem*{remark}{Remark}
\newtheorem*{claim}{Claim}

\theoremstyle{definition}
\newtheorem*{definition}{Definition}

 \divide\oddsidemargin by 2
\input{macros.tex}
\renewcommand{\hat}{\widehat}

\newcommand{\s}{{\bf{s}}}

\newcommand{\BBt}{{\widetilde{\BB}}}

\renewcommand{\De}{\Delta}

\newcommand{\dde}{\partial\Delta}
\newcommand{\Deh}{\hat{\Delta}}
\newcommand{\deh}{\hat{\Delta}}
\newcommand{\Det}{\widetilde{\Delta}}
\renewcommand{\det}{\widetilde{\Delta}}

\renewcommand{\hat}{\widehat}

\newcommand{\Xhat}{\widehat{X}}
\newcommand{\Yhat}{\widehat{Y}}
\newcommand{\Dstar}{{\D^*}}

\newcommand{\fum}{f_U^{-1}}

 \title{Singular values and bounded Siegel disks}
 
\vspace{5cm}

 \author{\small Anna Miriam Benini \thanks{Partially supported by the ERC grant HEVO - Holomorphic Evolution Equations n. 277691.}\\  
\small CRM Ennio de Giorgi\\
\small Piazza dei Cavalieri 3 \\   
\small 56100 Pisa, Italy\\ 
\small {\tt ambenini$@$gmail.com} 
\and 
\small N\'uria Fagella\thanks{Partially supported by the Catalan grants 2009SGR-792/2014 SGR 555, by the Spanish grant MTM2011-26995- C02-02 and by Polish NCN grant decision DEC-2012/06/M/ST1/00168. }\\   
\small Dept.~de Mat.~Aplicada i An\`alisi\\ 
\small Univ.~de Barcelona, Gran Via 585 \\ 
\small 08007 Barcelona, Spain\\  
\small {\tt fagella$@$maia.ub.es} 
}

\begin{document}

\maketitle  
\begin{abstract}Let $f$ be an entire transcendental function of finite order and $\Delta$ be a forward invariant bounded Siegel disk for $f$ with rotation number in Herman's class $\HH$. {We show that if $f$ has two singular values with bounded orbit, then  the boundary of $\Delta$ contains a critical point. We also give a criterion under which the critical point in question is recurrent. We actually prove a more general theorem with less restrictive hypotheses, from which these results follow.}

%{\red [Remove if we remove last section] 
%In the appendix we prove several useful facts on the relation between preimages of sets and the presence of singular values.}
\end{abstract}

\section{Introduction}

We consider the dynamical system generated by the iterates of an entire transcendental function $f:\C \to \C$, that is a function which is holomorphic on the complex plane $\C$ and has an essential singularity at infinity (in general, we will omit the word 'transcendental'). In this setup, there is a dynamically natural partition of the phase space into two completely invariant subsets: the {\em Fatou set} $\FF(f)$, formed by those  $z\in\C$ for which the family of iterates $\{f^n\}_{n\in\N}$ is  normal in the sense of Montel in some neighborhood of $z$; and the {\em Julia set} $\JJ(f)$, its complement. Orbits in the Julia set exhibit chaotic behavior -- in fact, $\JJ(f)$ is the closure of the repelling periodic points of $f$.

The Fatou set is open, possibly with infinitely many components, called {\em Fatou components}. The periodic ones are completely classified into basins of attraction of attracting or parabolic cycles, Siegel disks (topological disks on which a certain iterate of $f$ is conjugate to a rigid irrational rotation of angle $\theta$, called the {\em rotation number}) or Baker domains (regions on which iterates converge uniformly to infinity).  Non-periodic Fatou components are  called \emph{preperiodic} if they are eventually mapped to a periodic component, and  \emph{wandering} otherwise. Baker and wandering domains are types of Fatou components which appear only in the transcendental setting. 

In holomorphic dynamics, the {\em singular values} of $f$ play a crucial role.  A singular value of $f$ is a point around which not all branches of $f^{-1}$ are well defined and univalent (see Section \ref{CoveringsSection}). {With this definition, the set
 $$
S(f):=\{s\in\C; \; s \text{ is a singular value for $f$}\}.
$$
 is closed. Observe  that  $f:\C\setminus f^{-1}(S(f))\ra \C\setminus S(f)$ is a covering of infinite degree.  A singular value $s\in\C$ is called  a \emph{critical value} if $s=f(c)$, where $c$ is a point with vanishing derivative or {\em critical point}. It is called     an \emph{asymptotic value} if there exists a curve tending to infinity whose image under $f$ is a bounded curve converging to $s$. Morally, asymptotic values concentrate infinitely many preimages at infinity.  The set of singular values consists of critical values, asymptotic values and any of their accumulation points.  {Singular values may also have {\em regular} preimages, i.e. preimages at which $f$ is locally univalent.} A singular value $s$ is called \emph{isolated} if there is a neighborhood $V$ of $s$ such that $V\cap S(f)= \{s\}$. 
 
Most types of Fatou components are associated with a \emph{singular orbit} (the orbit of a singular value). For example, all basins of attraction of attracting or parabolic cycles  must contain at least one  singular  orbit  \cite[Theorems 8.6 and 10.15]{Mi} which therefore accumulates on the periodic cycle. 

Siegel disks cannot contain critical points although they may contain asymptotic or critical values if these have regular preimages which are contained in the  cycle of Siegel disks.  On the other hand a theorem of Fatou \cite[Corollary 14.4]{Mi} asserts that the boundary of a Siegel disk must be in the accumulation set of the union of singular orbits.  The relation of $S(f)$ with wandering or Baker domains is not as well understood but, if  $f$ is an  entire function with finitely many singular values, neither wandering domains nor Baker domains exist (\cite{EL,GK}).  Likewise, if $f$ belongs to the Eremenko-Lyubich class
\[
\BB:=\{ \text{$f:\C\to \C$ entire transcendental;  $S(f)$ is bounded} \}, 
\]
 then $f$ has no Baker domains, nor any wandering domains for which the iterates converge uniformly to infinity \cite{EL}. Wandering domains which oscillate are possible \cite{bishop}, and wandering domains with  bounded orbit have not yet been discarded, nor proven to exist. %{\red Did you ever think about the fact that if there are only finitely many s.v. interacting with delta this is not possible? Can we do a surgery construction? Just a curiosity]}. 

%{\b An even stronger condition is requiring that the orbits of the singular values are bounded. The closed union of all {\em singular orbits} (the orbit of a singular value) is called the \emph{postsingular set} 
% \[
% \PP(f):=\ov{\underset{s\in S(f)}\bigcup f^n(s)}.
% \]
%We say that $f$ is {\em postsingularly bounded} is $\PP(f)$ is bounded. Postsingularly bounded maps have special good properties, somewhat analogously to polynomials with a connected filled Julia set. In particular, all periodic rays land \cite{}.
%}

From now on, let $\De$ be a {\em bounded} invariant Siegel disk with rotation number $\theta$.  The arithmetics of $\theta$ play a relevant role. The set of {\em Herman numbers} $\HH$ is defined as the set of $\theta \in \R$ such that every orientation preserving analytic diffeomorphism of the circle is analytically linearizable, that is, analytically conjugate to the rigid rotation $\RR_\theta$. This set was proven to be nonempty by Herman \cite{He1} who showed that it contains all diophantine numbers, and was later described arithmetically by Yoccoz \cite{Yo}. 

Although the topology of $\De $ is trivial and the dynamics is well understood, neither assertion is true for the boundary of the Siegel disk. Certain properties depend on the rotation number $\theta$, and in some instances, on the dynamics of $f$ outside of $\Delta$. For example, it is not known (even in rational or polynomial dynamics) whether $\dde$ is always a Jordan curve, although it is believed to be so. In transcendental dynamics, unbounded Siegel disks may have non-locally connected boundaries due to the presence of an asymptotic value in the boundary (using its preimage at infinity), as it occurs for example for some members of  the  exponential family $\lambda e^z$. In the case of bounded disks however, no example has been found  yet for which the boundary is not a Jordan curve. 

The existence of an asymptotic value (acting as such) or a critical  point in the boundary of a Siegel disk is an obvious obstruction for the linearization domain to be extended further and hence a natural reason for the maximality of $\Delta$. A theorem of Ghys \cite[Theorem 3]{ghys} states that if  $f$ has rotation number $\theta \in \HH$ on an invariant  Jordan curve $\gamma$ and $f$ is univalent in a neighborhood of $\gamma$, then  $f$ is analytically conjugate to the rigid rotation $R_\theta$ in an annular neighborhood of $\gamma$.  The main corollary is that if $\theta \in \HH$ and $\partial \Delta$ is a Jordan curve then, independently of the nature of $f$ outside the disk, $\partial \Delta$ contains a critical point. The proof of this theorem is at the base of many later results related to the problem of existence of critical points on the boundary of Siegel disks, which is the main problem addressed in this paper.  Finally observe that  that the arithmetic assumption on $\theta$ is non removable, as shown by the examples of Siegel disks with Jordan boundaries without singular values constructed by  Ghys \cite[Theorem 5]{ghys}, Herman \cite{her3}  and  Rempe \cite[Theorem 6.1.4]{remthesis} for exponential maps. 

The question remains of whether any Siegel disk $\Delta$ with rotation number $\theta \in \HH$ contains a critical point (or an asymptotic value) on its boundary, {\em without any assumption on the topology of $\partial \Delta$}.
 
 This question was first addressed by Herman \cite{He2} who, based on Ghys' ideas, developed a general theorem about extensions of holomorphic maps which leave invariant a rotation annulus with rotation number in $\HH$. His result and some interesting corollaries are summarized in the following statement. 

\begin{theorem}[Herman]\ 
\begin{itemize}
\item[\rm (a)]
Let $f$ be  a holomorphic map on a domain $U\subset\C$, and $\Delta$ be a bounded invariant Siegel disk of $f$ with rotation number $\theta\in\HH$ such that $\ov{\Delta}\subset U$. 
Then there is a critical point in $\dde$ provided that moreover $f|_{\dde}$  is injective.
\item[\rm (b)] If $f(z)=z^d + c$ with $d\geq 2$ and $c\in\C$ has a Siegel disk $\Delta$ with rotation number $\theta \in\HH$, then $\partial \Delta$ contains a critical point.
\item[\rm (c)] If $f(z)=\lambda e^z$ with $\lambda \in\C\setminus\{0\}$ has a Siegel disk $\Delta$ with rotation number $\theta \in\HH$, then $\Delta$ is unbounded.
\end{itemize}
\end{theorem}

Hence part (b) gives an answer for unicritical polynomials. Very recently Ch\'eritat and Roesch \cite[Theorem 1]{CR} went one step further and proved that Siegel disks of polynomials with two critical values and rotation number in $\HH$ have a critical point on the boundary. They conjecture that this will be the case for all rational maps.

\begin{theorem}[Ch\'eritat-Roesch]\label{CR} Let $P$ be a polynomial with two finite critical values, and $\Delta$ be an invariant  Siegel disk with rotation number in $\HH$. Then   $\partial\Delta$ contains a critical point. 
\end{theorem}

Part (c) of Herman's theorem deals with the exponential family, which has one asymptotic value and no critical points.  Rempe \cite{rem04}  showed additionally that such unbounded Siegel disks must contain the asymptotic value on the boundary. He later considered a more general setting in \cite{rem08} where he showed that if an entire transcendental map has a Siegel disk $\Delta$ with rotation number in $\HH$, then  the boundary of $\Delta$ contains a {\em singular value}, provided that $S(f)\subset J(f)$. 

Inspired by Ch\'eritat-Roesch Theorem, our goal  is to address the problem for entire transcendental  functions  having a {\em bounded} invariant Siegel disk $\Delta$. Intuitively, if the Siegel Disk is bounded and the rotation number is sufficiently well behaved, the Siegel disk interacts with the critical values as in the rational case, and not with the asymptotic values as in a typical transcendental case. Nevertheless,  although in some cases we can indeed reduce locally to the polynomial setting, in general we must take into account the very different nature of $f$ provided by its essential singularity  at infinity. The infinite degree and the presence of asymptotic values give rise to more possibilities than those encountered in the polynomial setting

Our main result (in its weakest form) is the following:

\begin{mainthm} Let $f$ be an entire function of finite order with at most two singular values and with bounded singular orbits. Suppose $f$ has a (bounded) invariant Siegel disk $\Delta$ with rotation number in $\HH$.  Then there is a critical point on $\dde$.
\end{mainthm}

 We do not prove this theorem directly. Instead, we state and prove a quite more general result from which the above will follow. The Main Theorem in its general form has somewhat more technical hypothesis which we explain  in the next section.  In particular the hypothesis of bounded singular orbits is much stronger than what is needed, but easier to state. A neither optimal alternative is to require that all periodic rays land (see Section~\ref{separation}).

Along the way we also prove  a criterion to establish when the critical point in question  is recurrent (see Proposition 1).

%%%%%%%
\subsection{Statement of results}

To state the general form of our results we first introduce some terminology. In what follows, $f$ is an entire transcendental map and $\Delta$ a bounded Siegel disk of $f$. 
Let   $\ov{\De}$ denote the closure of $\Delta$ and $\hat{\De}$ be the complement  of the unique unbounded  connected component of $\C\setminus\ov{\De}$. {Observe that $\Deh$ is full by definition.} If the boundary of $\De$ is a Jordan curve,  then $\deh=\ov{\De}$. 

Because $\De$ is forward invariant, $\dde$ also is. The bounded components of $\C\setminus\ov{\De}$ are called  \emph{hidden components} and are denoted by  $H_i$. 
Hidden components (see \cite{He2})  are Fatou components, because $\partial H_i\subset\dde$ and $f(\dde)=\dde$; so by the maximum principle the family of iterates on $H_i$ is bounded hence normal.  For the same reason  the images of hidden components are contained in $\Deh$, so  $\deh$ is forward invariant.  As a consequence, it is easy to check that $\partial \Delta = \partial \Delta_\infty$ where $\Delta_\infty$ denotes the unbounded component of $\C\setminus \widehat{\Delta}$. This striking fact does not imply a priori that there are no hidden components, as shown in the lakes of Wada constructions (see \cite{Yn}).  Examples of Siegel disks with hidden components have never been found; however, their existence has not been discarded either.

In general we allow $f$ to have infinitely many singular values, although in many cases only a finite subset is allowed to interact with $\Delta$. 
\begin{definition} (Interacting singular values) 
We say that only a finite set of singular values $S_\Delta \subset S(f)$ is {\em interacting with} $\Delta$ if the following conditions are satisfied:
\begin{enumerate}
\item[\rm (a)] $\omega(s)\cap\deh \neq \emptyset$ for all $s\in S_\Delta$;
\item[\rm (b)] $\ov{\bigcup_{s\in S(f)\setminus S_\Delta}\omega(s)}\cap\deh=\emptyset$, and
\item[\rm (c)] the  singular values in $S(f) \setminus S_\Delta$  do not accumulate on $\dde$.
%\item[\rm (c)] $\ov{S(f) \setminus S_\Delta}\cap\dde=\emptyset$
\end{enumerate}
\end{definition}
If only a finite set $S_\De$ of singular values is interacting with $\De$, then  the singular values in $S(f)\setminus S_\Delta$ are called {\em non-interacting}.
Escaping singular values, as well as singular values which belong to preperiodic Fatou components whose closure is disjoint from $\dde$, are non-interacting (if they are in wandering domains, they could).  Also note that if $S(f)$ is finite then it is (trivially) true that there are  only finitely many singular values interacting with $\De$. 

A property that plays a role in a certain part of the proof is the following. 
\begin{definition}[Separation Property]
We say that a holomorphic function $f$ has the {\em Separation Property} if the closures of any two periodic components of $F(f)$ may intersect in at most one point.
\end{definition}

The Separation Property is satisfied by all polynomials with a connected Julia set  \cite{GM}. Entire maps for which an analogous of external rays can be defined and such that all periodic rays land, also have this property \cite{BF} (see Section \ref{separation} for details). 
In particular, functions of finite order in class $\BB$ (or compositions thereof) possess {\em dynamic rays} in their escaping set \cite{R3S}.  It is expected that, unless a periodic ray contains  a singular value, it does land, in which case it necessarily has to land at a repelling or parabolic periodic point by the Snail Lemma. This is the case, for example, in polynomial dynamics (\cite{Mi}, Theorem 18.10), and for the exponential family \cite{rem06}.
In the general case, however, the simplest (and strongest) criterion for all periodic rays to land is to assume that the union of all singular orbits or {\em postsingular set} is bounded.

%and whenever the postsingular set 
%\[
%P(f):=\overline{\bigcup_{s\in S(f)} f^n(s)}
%\]
%is bounded, all periodic rays land \cite{rem08,Fa,De}. 
%  
Our  main theorem in the most general form reads as follows.

\begin{mainthm2} \label{Main theorem} 
Let $f$ be an entire function with no wandering domains and satisfying the Separation Property. Suppose $f$ has  a bounded invariant Siegel disk $\Delta$ with rotation number $\theta \in \HH$, and such that  only two  singular values are interacting with $\Delta$.  Then there is a critical point on $\dde$.
\end{mainthm2}

\begin{remark} 
The hypothesis of no wandering domains can be replaced by the stronger requirement for $f$ to have finitely many singular values since, as we mentioned above,  such functions are  known not to have wandering domains \cite{EL,GK}.
\end{remark}

%Since by \cite{EL} functions with finitely many singular values do not have wandering domains, and using   the aforementioned results from  \cite{Re1} (see also \cite{De}) to ensure that Fatou components can be separated  we obtain the following corollary.

An immediate corollary is the following:
\begin{corollary}
Let $f$ be an entire function with no wandering domains and satisfying the Separation Property. 
If $f$ has no critical values then all Siegel disks with rotation number in $\HH$ are unbounded. 
\end{corollary}
%{\b The Main Theorem also implies  that, for  a function satisfying our hypothesis and having no critical values, all Siegel disks with rotation number in Herman's class are unbounded.}

In view of the results in \cite{BF,remthesis} mentioned above, we recover Herman's Theorem (c) which says that all Siegel disks of the exponential map with rotation number in $\HH$ are unbounded. 

%. Indeed, if an exponential map has a Siegel disk, its unique singular value (an asymptotic value) cannot belong to a ray and  hence all periodic rays land \cite{remthesis}. This means that the map satisfies the Separation Property \cite{BF} and therefore any bounded Siegel disk with rotation number  in $\HH$ must contain a critical point on the boundary which is impossible.

The proof of the Main Theorem is subdivided into three cases, each of them  requiring different hypothesis: In  case A we assume that all singular values interacting with $\De$ are outside $\Deh$, in case B we assume that they are all contained in $\Deh$, and in case C we assume that there is exactly one interacting singular value inside and one outside of $\Deh$. This is the same proof structure that was used in the polynomial case in \cite{CR}. Our contribution in this paper consists of the proof of case  B without using  Ma\~ne's theorem (see below)  and the proof of case C when there are asymptotic values interacting with the Siegel disk. To make the paper reasonably  self-contained,  we include the main ideas of the full proof.
We remark that the Separation Property and the restriction to two singular values is only necessary for case C.   

We now  list  the three independent statements here named as Propositions A, B and C.

\begin{propA}[All singular values outside]
Let $f$ be an entire map with a bounded Siegel disk $\De$ of rotation number $\theta$, and suppose  $S(f)\cap \Deh=\emptyset$. Then 
\begin{itemize}
\item[(a)]  $\theta\notin\HH$  (Herman);
\item[(b)] If furthermore $f$ satisfies the Separation Property and has no wandering domains, then $\De$ has no hidden components. 
\end{itemize}
\end{propA}

The first statement is due to Herman as an easy consequence of his theorem stated above (part (a)).  The second statement is actually an addendum since it does not contribute to the proof of the Main Theorem. It is in fact a transcendental counterpart of  \cite[Theorem 3.1]{Ro} and is new
in the sense that until now there were no known functions satisfying the Separation Property (see Section {\ref{separation}).  What is actually shown is that any hidden component must be a wandering domain with bounded orbit. Wandering domains with bounded orbit are not known to exist, but not discarded either. It is plausible that they could be excluded under the asumption of having only finitely many singular values interacting with $\De$.

We remark that the nonexistence of hidden components does not imply that $\partial \De$ is a Jordan curve, as shown by the Siegel disk examples of Ch\'eritat in \cite{che}, whose boundary is a pseudocircle (thus separating the plane into two components but not being locally connected at any point). These examples however, do not come a priori from a globally defined map. 

%.  The proof of Proposition A is  not really new, however part (b) (which is not used in the proof of the Main Theorem) is a new result since there was  previously no known condition under which periodic  Fatou components could be separated (see \cite{BF}). 

%We do not known any examples of Siegel disks whose hidden components are wandering domains; probably this can be excluded under the assumption that there are only finitely many singular values interacting with $\De$, however, the time being, the assumption that there are no wandering domains in $\Deh$ is necessary.

The case in which all singular values are contained in $\Deh$ is treated in Proposition B.

\begin{propB}[All singular values inside]
Let $f$ be an entire map with a bounded invariant Siegel disk $\De$. Suppose that  there are only finitely many singular values interacting with $\De$, all of which are contained in  $\Deh$, and that there are no wandering domains in $\Deh$. Then there is a {\bf recurrent} critical point on $\dde$.
\end{propB}

In \cite{CR}, the analougous case for polynomials is proved using Ma\~ne's Theorem \cite{Ma3}  which ensures  the existence of a recurrent critical point whose orbit accumulates on $\partial \De$. The transcendental version of this nontrivial result is not yet proven in full generality (see \cite{RvS} for a certain class of entire maps). Instead, our proof  of Proposition B relays on a variation of Fatou's Theorem (see  Lemma~\ref{Local Mane}), that which asserts that all points in the boundary of a Siegel disk must be accumulated by the union of the singular orbits. In the process, we extract the following corollary which we believe might be useful in itself.  

\begin{prop1}[Recurrent critical points]%\label{Fatou variation}
Let $f$ be an entire function with a bounded invariant Siegel disk $\De$. Suppose that  there are only finitely  many singular values interacting with $\De$, all of which are contained in  $\dde$. Then there is at least one recurrent critical point on $\dde$.
\end{prop1}}

The final case in which there is exactly one interacting singular value in $\Deh$ and one outside $\Deh $ completes the proof of the Main Theorem for functions with only  two singular values interacting with $\De$.

\begin{propC}[One inside one outside] Let $f$ be an entire transcendental map  with no wandering domains which has a bounded invariant Siegel disk $\De$ with rotation number $\theta\in\HH$. 
Suppose that $f$ has only two {singular values $v,v'$ interacting with $\De$}, and satisfies the Separation Property. If there is exactly one singular value $v$ in $\Deh$, then  $v\in \partial\Delta$ and  there is a critical preimage $c$ of $v$ on $\dde$.
\end{propC}

The Main Theorem in its simplified version follows from the Main Theorem in the general form  by using the following facts: if there are finitely many singular values there are no wandering domains \cite{EL,GK}, and under the assumption of finite order and two singular values with bounded orbits, $f$ satisfies the Separation Property (see the comments above the Main Theorem or Section \ref{separation}). So from now on we will only take care to prove the Main Theorem in its general form and, in particular, Propositions A, B and C.

\subsection*{Acknowledgments}
We are very thankful to Arnaud Ch\'eritat and Pascale Roesch for sharing with us  the ideas of their work on bicritical polynomials, as well as their manuscript which led to this present work.
We also thank  Walter Bergweiler and Lasse Rempe-Gillen for useful comments about the manuscript. The first author was partially supported by the ERC grant HEVO - Holomorphic Evolution Equations n. 277691.

%%%%%%%%%%%%%%%%%%%%%%%
\section{Preliminaries and tools}

\subsection{Mapping properties of entire functions}\label{CoveringsSection}

{Recall from the introduction that a singular value is a point for which there is no neighborhood on which all branches of the inverse function are well defined and univalent. 
In this sense, such a value represents a singularity for $f^{-1}$, although the latter is globally defined only as a multivalued function. 
A singular value which is not accumulated by other singular values is called an \emph{isolated singular value}, and points in $\C\setminus S(f)$ are called \emph{regular points}. 
   Singular values  can also have noncritical preimages: if $a$ is a preimage of $s$ such that $f$ is univalent in  a neighborhood of $a$, we say that $a$ is a \emph{regular preimage} of $s$; otherwise, we say that it is a \emph{critical preimage}. 
}

{For  the following classification of {singularities of $f^{-1}$} see  \cite{BE} and \cite{Iv}.}

 \begin{prop}[Classification of singularities]\label{Singularities}
Let $f$ be an  entire function, $z\in\C$, $D(z,r)$  be the disk of radius $r$ centered at $z$. For any $r$ let $U_r$ be a component of $f^{-1}(D(z,r))$, chosen such that $U_r\subset U_{r'}$ if $r<r'$. Then only two cases are possible:
\begin{itemize}\item[(a)] $\bigcap_r U_r=\{p\}, p\in\C$
\item[(b)] $\bigcap_r U_r=\emptyset$
\end{itemize}  
In case $(a)$,  $f(p)=z $, and either $f'(p)\neq 0$ {and $p$ is a regular preimage of $z$}, or $f'(p)= 0$ and $z$ is a critical value.
In case $(b)$,   the chosen branch $U_r$ defines a transcendental singularity over $z$, $z$ is  an asymptotic value for $f$,  and we  say that $z$ has a \emph{preimage at infinity} along that given branch. 
\end{prop}

{Singular values can be both critical values or  asymptotic values  (depending on the selected branch of the inverse).   A singular value can be an accumulation point of asymptotic or critical  values without being asymptotic or critical itself:  in this case, it is a  singular value in the sense that not all branches of the inverse are well defined and univalent in any  neighborhood, but all its preimages could be regular in the sense of Proposition~\ref{Singularities}. An example of such a point is the value $1$ for the function $f(z)=\frac{(z+c)\sin(z)}{z}$ for appropriate values of $c$.}

The next lemma is a  basic fact in algebraic topology, see for example \cite{Ha} for general theory about coverings. 
\begin{lem}[Coverings of $\D$ and $\D^*$]\label{Coverings}
 Let $U\subset\hat{\C}$, $\D$ be the unit disk and $\D^*=\D\setminus\{0\}$.
 \begin{itemize}
\item[(a)] If $f$ is a holomorphic  covering from $U\ra \D$, then $U$ is simply connected and  $f$ is univalent.
\item[(b)] If $f$ is a holomorphic  covering from $U\ra \D^*$, then either $U$ is biholomorphic to $\D^*$  and $f$ is equivalent to $z^d$, or $U$ is simply connected and $f$  is the universal covering, hence equivalent to the exponential map.
 \end{itemize}
\end{lem}

\begin{proof}[Sketch of proof]

 As $f$ is a covering, the fundamental group of $U$ is a subgroup of the fundamental group of $f(U)$ which is either $\D$ or $\D^*$.
 \begin{itemize}
\item[(a)]  The fundamental group of $f(U)=\D$  is trivial, hence $U$ is simply connected as well and $f$ is a homeomorphism. 

\item[(b)]The fundamental group of $U$ is a subgroup of the fundamental group of $\D^*$ which is $\Z$, hence it is either trivial or isomorphic to $\Z$. In the first case $f$ is the universal covering and $f$ is equivalent to the exponential map. In the second case, $U$ is the quotient of the upper half plane by a translation, hence  $U\sim\D^*$  and $f$ is equivalent to $z^d$. 
\end{itemize}
\end{proof}

\begin{cor}[Separating branches] \label{separatingbranches} 
If $f$ is a holomorphic map and  $z\in\C$ is an isolated singular value for $f$, then   there is a neighborhood $V$ of $z$ such that each component of $f^{-1}(V)$ is simply connected, and is either  unbounded or containing at most one critical point. 
\end{cor}
\begin{proof}Let $V$ be any simply connected neighborhood of $z$ such that $V\cap 
S(f)=\{z\}$ which exists  because  $z$ is not  accumulated by singular values. Let $U$ be a component of  $f^{-1}(V)$. By  Lemma~\ref{Coverings}, either  $f:U\ra V$ is univalent, or $f:U\setminus \{f^{-1}(z)\} \ra V\setminus\{z\}$ is a covering of $\Dstar$. In this case,  either $U$ contains exactly one critical point and $f$ is equivalent to $z^d$, or $U$ is unbounded and  simply connected and $f$ is equivalent to the exponential. 
\end{proof}

%%%%%%%%%%%%
\subsection{The Separation Property} \label{separation}

In this section we want to give some details about the Separation Property and the class of functions which satisfy it. 

Let $\BB$ be the class of entire functions with bounded set of singular values, as defined in the introduction. We set 
\[
\BBt=\{ f=g_1\circ \cdots \circ g_k \mid g_i \in \BB \text{ and has  finite order}\}.
\]
For this class of functions, it is shown in \cite{R3S} that points whose orbits escape to infinity are organized in unbounded curves mapped to each other according to some symbolic dynamics. More precisely, there are maximal injective curves $g_{\s}:(0,\infty)\ra\C$, where $\s\in\Z^\N$  is a labelling, such that:
\begin{itemize}
\item $\underset{t\ra\infty}\lim |f^n(g_\s(t))|=\infty\,\  \forall n\geq 0$;
\item $\underset{n\ra\infty}\lim |f^n(g_\s(t))|=\infty$  uniformly in $[t,\infty)$.
\end{itemize} 
 These curves are called \emph{dynamic rays}, or just \emph{rays}, and are in many ways  analogous to external rays of polynomials. If $\lim_{t\ra0}g_{\s}(t)$ exists, it is said that the ray $g_\s$ \emph{lands}. A  ray  $g_\s$ is called \emph{periodic} if for some $k>0$, $f^{k}(g_\s)\subset g_\s$.

Although it is plausible to think that periodic rays land unless they contain a singular value, for the time being this is only proven for  polynomial or exponential dynamics. A sufficient condition  for a periodic  ray $g_\s$ to land (see \cite{Fa,rem08}) is that  $g_s\cap P(f)=\emptyset$, where $P(f)$ denotes the postsingular set  
\[
P(f):= \ov{\bigcup_{s\in S(f)} \bigcup_{n\geq 0} f^n(s)}.
\]
Hence to ensure that all periodic rays land we may require, for example, that the postsingular set is bounded (see e.g. \cite{De}), although this is a very strong condition.

% For functions of finite order this condition it is implied by a variety of conditions, of which is to assume, for example, that  (see \cite{Re1},\cite{Fa} and Prop.~\ref{Separating Fatou components}). 

Under the assumption that $f\in\BBt$ and that periodic rays land, it is shown in \cite[Corollary D]{BF} that $f$ satisfies the Separation Property. More precisely the following result is proven.

\begin{thm}[Separation Theorem]\label{Separating Fatou components} If $f\in\BBt$ and periodic rays land,  any two periodic Fatou components can be separated by two periodic rays landing at the same point.   Hence the boundaries of any two bounded Fatou components intersect in at most one point. 
%Moreover there are no Cremer points on the boundary of periodic Fatou components.
\end{thm}
%So whenever we refer to the condition that \emph{periodic Fatou components can be separated} we mean that $f\in\BBt$ and that periodic rays land.
A consequence of Theorem~\ref{Separating Fatou components} is the following corollary \cite[Corollary E]{BF}, which is directly related with our problem at hand. 

\begin{cor}[Preperiodic hidden components]\label{Preperiodic hidden Components}
If $f$ is an entire function satisfying the Separation Property then  any hidden component of a bounded forward invariant Siegel disk is preperiodic to the Siegel disk itself or a wandering domain with bounded orbit. 
\end{cor}
The corollary follows quite directly from the Separation Theorem. If a hidden component $H$ is  periodic, its boundary can intersect $\partial \De$ in at most one point. But $\partial H \subset \partial \De$ and hence this is impossible.

%%%%%%%%%%%
\subsection{Properties of  hats}

Given a bounded set $X\subset \C$, recall that $\widehat{X}$ denotes the union of $\ov{X}$ and the bounded components of $\C\setminus \ov{X}$.  (By the name of component we always mean connected component.) Observe that $\ov{X}$ is always closed and full. 
 
The following proposition summarizes the facts that we shall use about ``hats''. They will only be used in the proof of Proposition C, in Section \ref{proofC}. Most of them can be found or deduced from results in \cite[Paragraph 4.11.1]{CR}, but we include the proof here for completeness and self-containment. 

\begin{prop}\label{Properties of hats}
Let $X, Y$ be  bounded subsets of $\C$,  $V$ be a simply connected neighborhood  of $\ov{X}$. Then the following facts hold:
\begin{enumerate}
\item[\rm (1)] $Y\subset \Xhat \Longleftrightarrow \ov{Y}\subseteq \Xhat\Longleftrightarrow \Yhat\subseteq \Xhat$.
%\item $f({X})\subset {X}\Rightarrow f(\Xhat)\subset \Xhat$.
%\item $\partial Y\subset \partial X \Rightarrow \Yhat\subset\Xhat$ {\red maybe no need}
\item[\rm (2)] Let $g$ be a univalent function defined in $V$. Then $\hat{g(X)}=g(\Xhat)$. 
\item[\rm (3)]  Let $f$ be an entire function without singular values in $V\setminus\Xhat$. 
Let $U$ be a bounded connected component of $f^{-1}(V)$. For a set $A$ compactly contained in $V$ we set $f_U^{-1}(A)=f^{-1}(A)\cap U$.  Then 
$$\widehat{f_U^{-1}(X)}= f_U^{-1}(\Xhat). $$
\end{enumerate}
\end{prop}

\begin{proof} \ \ 

\noindent (1)  This follows because $\Xhat$ is both closed and full.
%\item Let $U$ be the component of $f^{-1}(V)$ containing $X$. Then $f:U\ra V$ is proper hence $f(\partial X)= \partial f(X)$. $f(\Xhat)$ is the union of $f(X), \partial f(X)$ and the image of the bounded components of $\C\setminus \ov{X}$. The latter ones are bounded components of $\C\setminus \ov{f(X)}$ giving the claim. 

\vspace{0.3cm}
\noindent (2) Observe that 
\begin{align*}
\hat{g(X)}&=\ov{g(X)}\cup \left\{\text{bounded components of $\C\setminus\ov{g(X)}$}\right\}\\
\Xhat&=\ov{X}\cup \left\{\text{bounded components of $\C\setminus\ov{X}$}\right\}.
\end{align*}
We first show that 
 $\hat{g(X)}\subseteq g(\Xhat)$. If $z\in \ov{g(X)}$, $g^{-1}(z)\in\ov{X}$ { because $g$ is a homeomorphism}, hence     $g^{-1}(z)\in\Xhat$ and $z\in g(\Xhat)$. If $z$ is in a bounded component of $\C\setminus\ov{g(X)}$ then $g^{-1}(z)$ belongs to a bounded component of $\C\setminus\ov{X}$  by the { maximum principle} hence $g^{-1}(z)\in \Xhat$ and $x\in g(\Xhat)$.

We now show that  $\hat{g(X)}\supseteq g(\Xhat)$. If $z\in g(\Xhat)$ then $g^{-1}(z)\in \Xhat$.  If $g^{-1}(z)\in \ov{X}$ then $z\in g(\ov{X})=\ov{g(X)}$ because $g$ is a homeomorphism, hence $z\in \hat{g(X)}$. Otherwise  $g^{-1}(z)$ is in a bounded component of $\C\setminus \ov{X}$ hence $z$ is in a bounded component of $\C\setminus \ov{g(X)}$ and $z\in \hat{g(X)}$.

\vspace{0.3cm}
\noindent (3)  Let $\{X_i\}_i$ be the components of $f^{-1}(X)$ which are contained in $U$, and $X'= \bigcup_i X_i=f_U^{-1}(X)$.
 By definition $f(X')=X$.  Since $\fum(X)$ has a finite number of components, $\partial X'=\bigcup_i \partial X_i$.  Since $f:U\to V$ is open and proper  it follows that $\ov{X'} \subset U$, $f(\ov{X'})=\ov{X}$ and $f(\partial X')=\partial X$. Consequently, since $f$ has no poles, by the maximum principle $f$ maps bounded components of $\C\setminus \ov{X'}$ to bounded components of $\C\setminus \ov{X}$.  Conversely, if $W$ is a bounded component of $\C\setminus \ov{X}$, then $f_U^{-1}(W)$ consists of a finite number of bounded components of $\C\setminus \ov{X'}$.
  
% {\b[Proof of $f(\ov{X'})=\ov{X}$ and $f(\partial X')=\partial{X}$.
%Observe first that since $f: U\ra V$ is proper, $f(\partial U)=\partial V$, hence if $\ov{X}$ is compactly contained in $V$ also $\ov{X'}$ is compactly contained in $U$.

 %1)$f(\ov{X'})\supset \ov{X}$. Let $y\in \partial X$, $y_n\in X, y_n\ra y$. Let $x_n\in X'$ such that $f(x_n)=y_n$, then by continuity for any of their accumulation points $x$ (which is not in $\partial U$ by properness of $f$) we have that $f(x)=y$. It is only left to show that $x$ is not an interior point for $X'$. But if it were so, there would be a small ball around $x$ which is  contained in  $X'$, hence since $f$ is open, there would be a small ball near $y$ which is contained in $X$, contradicting  $y\in \partial X$. 
%The reverse inequality follows from continuity of $f$ and the fact that we took closures.

%2) $f(\partial X')=\partial{X}$. Proof of $\supset$ goes the same way as before. For the reverse argument take any $x\in \partial X'$ and any neighborhood $W$ of $x$. Then $W$ intersects $X'$ and $U\setminus X'$. Now $f(W)$ is a neighborhood of $y=f(x)$ which has points of $f(X')=X$ and $f(W\setminus X') = V\setminus X$.
 
% }
  
 We first show that $\hat{X'}\subseteq \fum{(\Xhat})$.  If $z\in \ov{X'}$, $f(z)\in\ov{X}$  hence $f(z)\in\Xhat$ and $z\in \fum(\Xhat)$ by definition. If $z$ is in a bounded component of $\C\setminus\ov{X'}$ then  $f(z)$ belongs to a bounded component of $\C\setminus\ov{X}$ hence $f(z)\in \Xhat$. Then by definition $z\in\fum(\Xhat)$. 

We now show that $\hat{X'}\supseteq \fum{(\Xhat})$. If $z\in \fum{(\Xhat})$ then $f(z)\in \Xhat$.   If $f(z)\in \ov{X}$,  then $z\in \ov{X'}\subset \hat{X'}$. Otherwise  $f(z)$ is in a bounded component of $\C\setminus \ov{X}$, hence $z$ is in a bounded component of $\C\setminus \ov{X'}$.
\end{proof}

%\subsection{Singularities of entire functions}\label{Singularities of entire functions}

%%%%%%%%%%%%%%%%%%%%%%
\section{Proof of Propositions  A, B and Proposition 1}\label{AB}

Let $\De$ be a bounded forward invariant Siegel disk of $f$, and $H_i$ the hidden components, that is  the bounded components of $\C\setminus\ov{\De}$. By definition, 
\[
\deh=\dde\cup\De \cup \bigcup_i H_i.
\]

Let $\Det$ be the component of $f^{-1}(\Deh)$ containing $\Deh$. 

The proof of part (a) in Proposition A is due to  Hermann (\cite{He2}); we sketch it  for completeness and refer to  \cite[Paragraph 4.2]{CR} for details. We give  a proof of part (b).

\begin{proof}[Proof of Proposition A] 
Since $\Deh$ {is full and }contains no singular values, and $S(f)$ is a closed set, there exists   a {simply connected } neighborhood $V$ of $\deh$ such that $\ov{V}\cap S(f)=\emptyset$. Let $U$ be the component of $f^{-1}(V)$ containing $\det$. 
 \begin{itemize}
\item[(a)]  By Lemma~\ref{Coverings}, $f:U\ra V$ is a homeomorphism, hence $\Det=\Deh$. The complement of $\Deh$ can be uniformized to the complement of the unit disk $\ov{\D}$, and the uniformization map conjugates $f$  to an analytic map $g$ defined in a small annulus around $\ov{\D}$. Because $\Deh$ is locally backward invariant, using Schwarz reflection's principle $g$ can be extended to an analytic map in a neighborhood of $\S^1$, and its restriction to $\S^1$ is an analytic circle map, which has degree  one because $f$ is univalent in a neighborhood of $\Deh$. The rotation number of $g$ can be shown to be the same as the rotation number of $\Delta$, i.e. $\theta$;  if $\theta \in\HH$, by Hermann's theorem \cite{He1} the map $g$ can be linearized via an analytic map. Then the  Riemann map transports it to a linearization for $f$ in a neighborhood of the Siegel disk, contradicting its maximality.
\item[(b)]
As in part (a), by Lemma~\ref{Coverings} the map $f$ from $U\ra V$ is univalent, hence the unique preimage of $\De$ in $U$ is $\De$ itself.
Using  Corollary~\ref{Preperiodic hidden Components} and the nonexistence of wandering domains,  any hidden component of $\deh$ is preperiodic to $\De$. So if a hidden component $H$ exists, there exists $n>0$ such that $f^n(H)=\De$ and  $f^{n-1}(H)\neq\De$. Hence  $f^{n-1}(H)$ is a preimage of $\De$ in $U$, which is a contradiction.
\end{itemize}
\end{proof}

We now proceed to prove Propositions B and Proposition 1. 

Proposition B relies strongly on the following lemma, which is a modification of a  theorem due to Fatou  asserting that the boundary of a Siegel disk is contained in the union of the $\omega$-limit set of the set of critical values (see Corollary 14.4  in \cite{Mi}). 

\begin{lem}\label{Local Mane} Let $f$ be an entire function with a bounded invariant Siegel disk $\De$. {Assume that there are only finitely many singular values interacting with $\De$.} 
 Suppose that  there exists a point $z\in\dde$ and a  neighborhood $V$ of $z$ such that $\omega(s)\cap V=\emptyset$ for all  singular values  which are in not in  $\dde$. Then there is at least one  recurrent critical point on $\dde$. 
\end{lem}

\begin{remark}
Under weaker hypothesis, it follows from  \cite[Corollary 2.9]{RvS}  that the the boundary of $\Delta$ is contained in the $\omega$-limit set of a recurrent or almost recurrent (see \cite{RvS} for a definition) \emph{singular value},  provided that the function in question has finitely many singular values.  But nothing is said (or easily deduced) about critical points. 
%Note also that our Lemma assumes that the singular values in question  all belong to $\partial \Delta$, in this way  excluding the case in which singular values accumulate on the boundary of $\Delta$ without actually belonging to it. [[Nœria: we do not actually assume this in the lemma]
\end{remark}

\begin{proof}

Let $v_1\ldots v_q\subset \dde$ be the finitely many singular values whose forward orbits intersect $V$. By hypothesis, it is possible to shrink $V$ (if necessary) to ensure that it contains no iterates of any singular value other than those of $v_1, \ldots, v_q$. For simplicity, let us also shrink it further so that at most one $v_i$ belongs to $V$.  Let $\epsilon$ be smaller than the  minimal distance between any two such singular values, and such that there are no other singular values in an $\epsilon$-neighborhood of $\dde$. 

We first show that there is at least one critical point in $\dde$. Suppose by contradiction that there are no critical points in $\dde$. Then  {all preimages of each $v_i$ which are contained in $\dde$ are regular and there exists at least one of them, since $\dde$ is bounded  and $f:\dde\ra \dde$ is surjective. Let $a_i \in \dde$ be the unique regular preimage of $v_i$ for which  the unique branch of $f^{-1}$ defined on  $D_i:=\D_\epsilon(v_i)$ mapping $v_i$ to $a_i$,  also maps $D_i\cap \Delta$ to $\De$. By injectivity of $f\mid_\De$, the point $a_i$ is unique for each $i$. We denote this special branch of $f^{-1}$ by $\psi_i$.}
\begin{claim}
Under the contradiction assumption, for each $n\in\N$ there is a unique  univalent  branch  $\phi_n$ of $f^{-n} $ defined on {$V$} which maps {$V\cap \De$} to $\De$.
\end{claim}
\noindent See also \cite[Corollary 14.4]{Mi}.   For $n=1$ the claim is true: %{\dg if no $v_i$ belongs to $V$. If one of them does, say $v_1\in V$, then choose $\psi_1$ to be the univalent branch of $f^{-1}$ that maps $v_1$ to $a_1$  and which therefore maps $V\cap \Delta$ to $\Delta$.}\r [for Nuria: the reason why we cannot  restrict $V$ so that it does not contain any singular values is that x could be a singular value itself? I think the cleanest thing to do is to state in the hypothesis of the lemma that x is a regular point and to restrict V so that it does not contain singular values. Observe that $v_1$ could have several regular preimage and that only one of them corresponds to the right branch mapping delta to delta.] Alternative justification: [Nœria: See note and modifications.]
 Since there is at most one   singular value, say  $v_1$, in $V$, the branch $\psi_1$ extends to $V$ and we may define $\phi_1$ to be this extension. 
 % and it has only regular preimages in $\partial \Delta$, by Lemma~\ref{Coverings} all preimages of $V$ intersecting $\Delta$ map univalently to $V$, and by injectivity of $f$ on $\De$ there is exactly one inverse branch which maps $V\cap \Delta$ to $\Delta$.

Now assume that for some $n$ the branch $\phi_n$ is well defined,  univalent and maps {$V\cap \De$} to $\De$  and consider $\phi_n(V)$, which is therefore simply connected. If there are no singular values in $\phi_n(V)$ the claim is evidently   true also  for $n+1$. 
 By construction, the only singular values that might belong to $\phi_n(V)$ are $v_1, \ldots , v_q$, since no other singular orbit points belong to $V$. Let us suppose that $v_1,\ldots,v_k \in \phi_n(V)$. Then, for each  $1\leq i \leq k$  consider the  inverse branch $\psi_i$ defined in  $D_i$ which maps  $\psi_i(v_i)=a_i$.  Each  branch $\psi_i$ can be extended univalently to any simply connected open subset of  $\phi_n(V)$ not containing any other singular value. 
%ATTENZIONE: BISOGNA PRIMA DIRE CHE MAPPANO DELTA TO DELTA, E DA QUI DEDURRE CHE COINCIDONO SU UN APERTO QUINDI SONO LA STESSA.  
Since $a_i\in \dde$ there is a sequence of points $d_n\in\De$ converging to $a_i$ such that {$f(d_n)\ra v_i$}, so by injectivity of $f$ on $\De$ and by the identity principle each extended $\psi_i$ is the unique inverse branch mapping $\De$ to $\De$ in its domain of definition. Since  any two such branches $\psi_i, \psi_j$ overlap on some open set,  by the identity principle they coincide and patch to a univalent inverse branch $\psi$ defined on all of {$\phi_n(V)$}. {It is now clear that $\psi$  is the unique branch of the inverse mapping $\phi_{n}(V)\cap\De$ to $\De$, so that   $\phi_{n+1}$ is well defined on $V$ and maps $V\cap \De$ to $\De$.  This concludes the proof of the Claim.}

%{\r It is only left to show that $\psi$  is the unique branch of the inverse mapping $\phi_{n}(V)\cap\De$ to $\De$, so that   $\phi_{n+1}$ is well defined on $V$ and maps $V\cap \De$ to $\De$. By  injectivity of $f$ on the Siegel disk and by the identity principle, it is enough to show that $\psi$ this is true on some connected component of $\De\cap \phi_n(V)$. But since $a_i\in\dde$, $\psi_i$ maps  $D_i   \cap\De$ to  $\De$ on the connected component of $\De\cap \phi_n(V)$ whose closure contains $a_i$ and the claim follows.}

 %Otherwise for any $v_i\in \phi_n(V)$ there is a unique univalent inverse branch $\psi_i$ defined in a disk $\D_\epsilon (v_i)\cap \phi_n(V)$ such that $\psi_i(v_i)=a_i$. Since $a_i\in\dde$,  $\psi_i$ maps at least a subset of $\De$ to $\De$, hence since  $f$ is injective on the Siegel disk every  $\psi_i$  coincides with the unique branch of the inverse mapping $\De$ to $\De$ (since they coincide on an open set) and the latter one can be extended univalently to all of $\phi_n(V)$. It follows  that $\phi_{n+1}$ is well defined on $V$.

The family $\{\phi_n\}$ is normal because it omits three values. Moreover it  does not converge to infinity, so up to passing to a subsequence we can assume that it converges to a univalent function $\phi$. As $\phi_n(V\cap\De)\subset\De$ for all $n$, it follows that $\phi(V\cap\De)\subset\De$ and $\phi$ is non-constant (because $\phi(V)\cap \JJ(f)\neq\emptyset$, and grand orbits of points in $\De$ do not accumulate on the Julia set), hence open. %{\r [Although Milnor says nothing, to see that it is constant, we probably would need to argue that the successive preimages of points in $\De$ cannot converge to the boundary. A priori, the sets $\phi_n(V)$ could be converging to a point in the boundary, like it would happen for an attracting basin, right? This does not happen because inside the Siegel disk, preimages of open sets do not shrink. Should we say something?] You mean non-constant?}
 Hence there is a point $y\in\dde {\cap \phi(V)}$ such that a neighborhood of $y$ is mapped into $V$ under infinitely many iterates, contradicting $y\in \JJ(f)$.  We therefore conclude that there exists at least one singular value, say $v_1$, whose preimage in $\partial \Delta$ is non-regular. If $v_1$ is an asymptotic value, then its preimage is at infinity and hence $\Delta$ is unbounded. It follows that $v_1$ is a critical value and its preimage in $\partial \Delta$ is a critical point, as we wanted to show.

We now show recurrence.  Suppose that there are $q$ critical points  $\{c_i\}$ on $\dde$ whose orbits accumulate on $z$, and let $\CC$ be the collection of these critical points.  None of them is contained in the accumulation set of {the orbit of} any singular value  $v\notin\dde$, for otherwise  the orbit of $v$ would also accumulate on $z$ contradicting the hypothesis on $z$.
So each  $c_i\in\CC$ satisfies the hypothesis of this lemma hence is in the accumulation set of the orbit of some (possibly different) critical point  $c_j\in\dde$;  but since the orbit of $c_j$ accumulates on $c_i$ whose orbit accumulates on $z$,  then  the  orbit of $c_j$ accumulates   on $z$ as well and hence  $c_j\in\CC$. 
Let $c_1$ be any of the critical points  in $\CC$, $c_2$ be a critical point whose orbit accumulates on $c_1$, $c_3$  be a critical point whose orbit accumulates on $c_2$ and so on.   Since 
% $\dde$ is bounded there are only finitely many critical points in $\dde$ hence in $\CC$, 
$\CC$ is finite,  $c_q=c_n$ for some $n<q$. But then iterates of $c_q=c_n$ accumulate on $c_{q-1}$ whose iterates accumulate on $c_{q-2}$  and so on until $c_n$, hence by extracting a diagonal sequence of iterates, iterates of $c_n$ accumulate on $c_n$ itself giving recurrence. This concludes the proof of Lemma 3.1.
\end{proof}

Proposition 1 in the introduction follows almost directly from Lemma \ref{Local Mane}. It is a weaker statement but somewhat easier to state. We are now ready to prove Proposition B.

\begin{proof}[Proof of Proposition B] Let $S_\Delta$ be the finite set of singular values interacting with $\De$. For any singular value  $s\notin S_\Delta$ we have that  $\omega(s)\cap\dde=\emptyset$, while for any  $s\in S_\Delta\cap F(f)$, by the classification of Fatou components and absence of wandering domains in $\Deh$,  either  $\omega(s)\cap\dde$ is empty or it consists of a finite parabolic orbit. 
So there is a point $z\in\dde$ which is not contained in $\omega(s)$ for any $s\notin\dde$ hence satisfying the hypothesis of Lemma~\ref{Local Mane}.
By this lemma, there is a recurrent critical point on $\partial \Delta$.  
\end{proof}

%\begin{remark}
%In fact, since under the hypothesis  of Proposition B all singular values interacting with $\De$ are also contained in $\Deh$, it follows from the proof of Proposition~\ref{Fatou variation} that the critical point on $\partial\Delta$ is recurrent. 
%\end{remark}

%%%%%%%%%%
\section{Proof of Proposition C} \label{proofC}

We now assume that $f$ has only two singular values $v,v'$ interacting with $\Delta$,  where   $\De$ is a bounded forward invariant Siegel disk for $f$ with rotation number $\theta\in\HH$. We also assume that periodic Fatou components can be separated and that $f$ has no wandering domains in $\Deh$. In this Section we  prove that if  there is exactly one singular value $v$ in $\Deh$, then there is at least one critical preimage $c$ of $v$ on $\dde$. Recall that $\widetilde{\Delta}$ denotes the component of $f^{-1}(\Deh)$ containing $\Deh$.

If $v$ is a critical value, our setup is equivalent to the one for polynomials and 
the proof is analogous. We reproduce it here for completeness,   referring to \cite{CR} only for a part of the argument which is quite technical and goes through directly in our case.
The new part is the proof in the case that  $v$ is an asymptotic value.

%\begin{lem}[Regular preimages]\label{Regular preimages} Let $f$ be an entire function, $\Delta$ a bounded Siegel disk for $f$. 
%Suppose that  $\SS(f)\cap\Deh=\{v\}$ such that $v$ is an isolated singular value and has a regular preimage in $\Det$.  Then 
%\begin{itemize}
%\item $f$ from a neighborhood of $\deh$ to a neighborhood of $\deh$ is univalent,  hence $\det$ equals $\deh$ and is bounded;
%\item $\theta\notin\HH$.
%\end{itemize}
%\end{lem}

The proof of Proposition C uses the following two results.

\begin{lem}[Regular preimages]\label{Regular preimages} Let $f$ be an entire function, $\Delta$ a bounded Siegel disk for $f$. 
Suppose that  $S(f)\cap\Deh=\{v\}$ where $v$ is an isolated singular value  with a regular preimage in $\Det$.  Then $\det=\deh$ and $\theta\notin\HH$.
\end{lem}

\begin{proof}
To show the first claim observe that since $\Deh$ is full and  $v$  is an isolated singular value, there is  a simply connected neighborhood $V$ of $\Deh$ such that  $\ov{V}\cap S(f)=\{v\}$; let $U$ be the component of $f^{-1}(V)$ containing $\Det$. The map $f:U\setminus f^{-1}(v) \ra V\setminus\{v\}$ is a covering, hence by Lemma~\ref{Coverings}, $f^{-1}(v)$ is either empty or a single point. By hypothesis it cannot be empty and  it is a regular preimage, thus $f:U\ra V$ is a homeomorphism and $\Det=\deh$.
The proof of Proposition A can then be repeated to show that $\theta\notin\HH$.\end{proof}

The proof of Proposition C uses a theorem of Ma\~n\'e about hyperbolicity of circle maps (\cite[Theorem A]{Ma1},\cite{Ma2}).

\begin{thm}[Hyperbolicity of circle maps]\label{Hyperbolicity of circle maps}
Let $g:\S^1\ra\S^1$ be a  $C^2$ map. Let $\Lambda\subset\S^1$ be a  forward invariant compact set which does not contain critical points or non-repelling periodic points; then either $\Lambda=\S^1$ and $g$ is topologically equivalent to a rotation, or $\Lambda$ is a hyperbolic set, that is there exist $\ov{k}, \eta>1$ such that for all $k>\ov{k}$ and all $z\in\Lambda$, $|(f^k)'(z)|>\eta$. 
\end{thm}

\begin{proof}[Proof of Proposition C]

Let $V\supset\Deh$ be a simply connected open set such that $V\cap S(f)=\{v\}$.  Let $U$ be the component of the preimage of $V$ which contains $\Deh$. As there is only one singular value $v\in V$ which has no regular preimage by Lemma~\ref{Regular preimages}, by Lemma~\ref{Coverings} either $f:U\ra V$ is equivalent to $z^d$ and $v$ has only one preimage in $U$ or $f:U\ra V$ is equivalent to the exponential map,  $v$ is an asymptotic value, and $\det$ is unbounded.

Recall that $\deh$ decomposes like
$$\deh=\dde\cup\De\cup\bigcup_i H_i,$$
and that $f:\dde\ra\dde$ as well as   $f:\De\ra\De$ are surjective.  Hence $v$ cannot belong to $\De$, or it would need to have a regular preimage in $\De$, which by  Lemma~\ref{Regular preimages} would imply that $\theta\notin\HH$.
If $v\in\dde$ by surjectivity of $f|_{\dde}$ it has a finite preimage  on $\dde$, hence $f:U\ra V$ cannot be equivalent to the exponential map.  Again by  Lemma~\ref{Regular preimages}, the preimage of $v$ in $\dde$   cannot be regular; 
so $v$ is a critical value and has a critical preimage on $\partial \Delta$ and there is nothing else to prove. 

So the only case left is the one in which   $v\in H^v$ for some hidden component $H^v$. In this case,  $v$ could be either critical or asymptotic but, by the discussion above $v$ has a unique preimage in $\Det$ and this preimage is not regular.  The rest of the proof is devoted to show that the condition  $v\in H^v$ leads to  a contradiction. %{\b TO change in future version: it is better to state the two cases as case 1 f:U to V equivalent to exp, and case 2: f: U to V equivalent to z^d because ininity is not really a preimage on deh}
 
 \vspace{0.2cm}
 
\noindent {\em Case 1: $f:U\ra V$ is equivalent to the exponential map}.  This is to say that $v$ is an asymptotic value. Necessarily $\det$ and hence $U$ are unbounded. Call $\De_i$ the components of  $f^{-1}(\De)$ which are contained in $U$. The $\De_i$'s are all preimages of the same bounded set under a holomorphic map, so they do not accumulate on any compact set. Thus showing that $\De_i \subset \Deh$ for infinitely many $i\in\Z$ would contradict  boundedness of $\deh$.
 
Since the map $f:U\ra V\setminus \{v\}$ is equivalent to the exponential map, and the map $z\ra e^z$ has a transitive infinite cyclic group of automorphisms which is isomorphic to $\Z$ (generated by translations by $2\pi i $), it follows that $f$ also has a transitive infinite cyclic group of automorphisms $G$ generated by some element $\rho \in G$. Up to labelling,  $\De_i=\rho^i \De$ for each $i\in\Z$. 
By Corollary~\ref{Preperiodic hidden Components}, $H^v$ is preperiodic to $\De$. Also $\deh$ is forward invariant, so $f^n(v)\in f^n(H^v) $ is mapped from one hidden component to another until it reaches $\De$. In particular, at the iterate just before entering $\De$,  the asymptotic value $v$ belongs to a direct preimage $\De_m$ of $\De$, which is also a hidden component of $\Deh$.  
By Proposition~\ref{Properties of hats} part $(1)$, $\De_m=\rho^m(\Delta)\subset \Deh\Rightarrow \hat{\rho^m(\Delta)}\subset \Deh$, and by Proposition~\ref{Properties of hats} part $(2)$ since $\rho^m$ is a homeomorphism, $\rho^m(\Deh)=\hat{\rho^m(\Delta)}$ so $\rho^m(\Deh)\subset\Deh$.

Now, for any $q\in\N$

$$ \rho^{mq}\Deh=  \rho^{m(q-1)}\rho^m\Deh \subset  \rho^{m(q-1)}\Deh\subset\ldots\subset\Deh  $$

But then, for any $q \in \N$, we have that $\rho^{mq}\De\subset \rho^{mq}\Deh\subset\Deh$. It follows that there are infinitely many $\De_i\subset\Deh$ contradicting boundedness of $\Deh$.

 \vspace{0.2cm}

{\noindent {\em Case 2: $f:U\ra V$ is equivalent to $z^d$.}  

   Then the  preimage of $v$ is a finite critical point $c\in \Deh$. Since $f$ is equivalent to $z^d$ for some $d\geq 2$ so it is proper, and  %up to taking a smaller $V'$ such that  $\deh\Subset V'\Subset V$ if necessary, we can assume that  
  $U$ (and hence $\det$) is bounded. 
 
This case is totally analogous to the polynomial case which was proven in \cite{CR}. The transcendentality of $f$ plays no role in this setup which can be solved with a local approach. Nevertheless we include the main ideas here for completeness.

The outline of the proof is the following: we first show  that $\Det=\Deh$, then uniformize $\C\setminus\Deh$ to find a conjugate map which extends to  an analytic circle map of  degree $d$. We show that the periodic cycles of the induced circle map are all repelling  by using the fact that if there existed a nonrepelling one,  the remaining singular value would have to  both be attracted to the nonrepelling cycle and accumulate on $\dde$ by Fatou's theorem, which is a contradiction.  We can then use Theorem~\ref{Hyperbolicity of circle maps} to deduce that $f$ has a  polynomial-like restriction in $U$ and conclude the proof using  Fatou's theorem.

The first goal is to show that $\Det=\Deh$. 
Since  $f$ is equivalent to $z^d$ it  has a cyclic symmetry group $G$ isomorphic to $\Z/d\Z$  generated by some element $\rho$.  
Like in Case 1, let $\Delta_i=\rho^i(\De)$ denote the preimages of $\De$ in $U$. By Proposition~\ref{Properties of hats} part $(3)$, $\Det=\hat{\bigcup \De_i}$. 
Using Corollary~\ref{Preperiodic hidden Components}, like in the case where $v$ is asymptotic,  we can find a direct preimage $\De_m$ of $\De$ which is contained in $\deh$. If  $c$ (the preimage of $v$ in $U$) is a simple critical point, this concludes the proof that $\Det=\Deh,$ because there are exactly two preimages $\De$ and $\De'$ of $\De$ and both are contained in $\Deh$. If $c$ has a higher multiplicity, the symmetry arguments are more involved: we  refer to  \cite[Proposition 26 and Paragraph 4.11.6]{CR} for the proof of the fact that $\Det=\Deh$.

Now, as in the proof of Proposition A,  the complement of $\Deh$ can be uniformized to $\C\setminus\overline{\D}$ via a holomorphic map $\Phi$, in order to get, by Schwarz reflection, an analytic map $g$ defined in a neighborhood of $\S^1$, such that $g|_{\S^1}$ is an analytic circle map  with no critical points (see \cite[Lemma 30]{CR}). As $g$ is conjugate to $f$ in a small neighborhood of the unit disk, using the fact that   $f: U\ra V$ has degree $d$, it can be shown that also $g$ has degree exactly $d$ on $\S^1$ \cite[Lemma 30]{CR}. 
Now suppose that $g$ has a  point $x$ of period $q$ in $\S^1$, with $|(g^q)'(x)|\leq 1$. Since $g$ preserves the unit circle, either $|g'(x)|<1$ or $g'(x)= e^{2\pi i \alpha}$ with $\alpha\in \Q$. In both cases there is an open set of points that tend to $x$ under iteration of $g$, whose pullback under $\Phi$ is an open set $B$ of points in $\C\setminus\deh$ tending to $\dde$ under iteration of $f$.
It follows that $B$ is a subset of a periodic Fatou component for $f$, which is either attracting or parabolic. Then the orbit of that periodic Fatou component  has to contain a singular value, which has to be $v'$ because  $f^n(v)\in\De$ for $n$ large enough. But then, since $f$ has exactly two singular values, it follows that all singular values are in the Fatou set, contradicting Fatou's Theorem because of the absence of wandering domains.

So there are no non-repelling periodic points on $\S^1$ and $g$ is hyperbolic on $\S^1$. Hence there is a neighborhood $U_\epsilon$ of $\S^1$ such that $g(U_\epsilon)\supset U_\epsilon$. Using $\Phi^{-1}$, the set $U_\epsilon\cap(\C\setminus\overline{\D})$ can be pulled back to the dynamical plane of $f$ to get a neighborhood $U_f$ of $\deh$ which is mapped outside itself under $f$, making the restriction of $f$ to $U_f$ a polynomial-like map with a unique critical point $c$. Then by Fatou's theorem, $\dde\subset\omega(c)$ and hence $c\in \JJ(f)$ (again by the absence of wandering domains), contradicting that $c\in H^v$.
}

\end{proof}

\end{document}

%% file: macros.tex
\newtheorem{thm}{Theorem}[section]

\newtheorem{cor}[thm]{Corollary}
\newtheorem{lem}[thm]{Lemma}

\newtheorem{prop}[thm]{Proposition}

\theoremstyle{definition}

\theoremstyle{remark}
\numberwithin{equation}{section}

\font\nt=cmr7

\def\note#1
{\marginpar
{\nt $\leftarrow$
\par
\hfuzz=20pt \hbadness=9000 \hyphenpenalty=-100 \exhyphenpenalty=-100
\pretolerance=-1 \tolerance=9999 \doublehyphendemerits=-100000
\finalhyphendemerits=-100000 \baselineskip=6pt
#1}\hfuzz=1pt}

\def\be{\begin{equation}}
\def\ee{\end{equation}}

\renewcommand{\epsilon}{\varepsilon}

\newcommand{\ra}{\rightarrow}

\newcommand{\De}{{\Delta}}

\newcommand{\BB}{{\cal B}}
\newcommand{\CC}{{\cal C}}

\newcommand{\FF}{{\cal F}}

\newcommand{\JJ}{{\cal J}}
\newcommand{\HH}{{\cal H}}

\newcommand{\RR}{{\cal R}}

\newcommand{\C}{{\mathbb C}}

\newcommand{\D}{{\mathbb D}}

\newcommand{\N}{{\mathbb N}}
\newcommand{\Q}{{\mathbb Q}}

\newcommand{\R}{{\mathbb R}}
\renewcommand{\S}{{\mathbb S}}

\newcommand{\Z}{{\mathbb Z}}

\def\B0{{\mathbf{0}}}

\newcommand{\ov}{\overline}

\renewcommand{\ra}{\rightarrow}

\catcode`\@=12

\def\Empty{}
\newcommand\oplabel[1]{
  \def\OpArg{#1} \ifx \OpArg\Empty {} \else
  	\label{#1}
  \fi}
		

%% file: SiegelDisksRev1.v7.bbl
\begin{thebibliography}{99}{ 

\bibitem[BF]{BF}A. M. Benini, N. Fagella, \emph{A separation theorem for entire functions}. Preprint (2011) ArXiv:1112.0531. 
%
%\bibitem[BeF]{BerF} R. Berenguel, N. Fagella, {\em An entire transcendental family with a persistent Siegel disk},  
%J. Difference Equ. Appl., {\bf 16} (2010), 523-553. 
%
%\bibitem[Be]{berper} W. Bergweiler, \emph{personal communication}.
%
\bibitem[BE]{BE} W.~Bergweiler,    A.Eremenko, \emph{On the singularities of the inverse of a meromorphic function of finite order}, Rev. Mat. Iberoamericana {\bf 11} (1995), no. 2, 355-373.
%
%\bibitem[BFR]{BFR} W.~Bergweiler, N. Fagella, L. Rempe-Gillen, \emph{Hyperbolic entire functions with bounded Fatou components}, manuscript (2014). 
%
\bibitem[Bi]{bishop} C. J. Bishop, {\em Constructing entire functions by quasiconformal folding}. Preprint. http://www.math.sunysb.edu/~bishop/papers/classS.pdf.
%
\bibitem[Ch]{che} A. Ch\'eritat, {\em Relatively compact Siegel disks with non-locally connected boundaries}, Mathematische Annalen {\bf 349} (2011), 529-542. 
%
\bibitem[CR]{CR} A. Ch\'eritat, P. Roesch, \emph{Herman's condition and Siegel disks of  polynomials}. Preprint (2011), arXiv:1111.4629v2.
%
\bibitem[De]{De} A.~Deniz, {\em A Landing Theorem for Periodic Dynamic Rays for Transcendental Entire Maps with Bounded Post-Singular Set}, Preprint (2014), ArXiv:1403.6598 [math.DS].
%
%\bibitem[D]{dou} A.~Douady, {\em Disques de Siegel et anneaux de Herman},
%Ast\'{e}risque, {\bf 152-153} (1987), 151--172.
%
\bibitem[EL]{EL} A. Eremenko, M. Lyubich, \emph{Dynamical properties of some classes of entire functions},  Ann. Inst. Fourier (Grenoble)  {\bf 42}  (1992),  no. 4, 989--1020. 
%
\bibitem[Fa]{Fa} N. Fagella, \emph{Dynamics of the complex standard family}, J. Math. Anal. Appl. {\bf 229} (1999), no.1, 1--31.
%

\bibitem[Gh]{ghys} E. Ghys, {\em Transformations holomorphes au voisinage d'une courbe de Jordan}, C.R. Acad. Sc. Paris {\bf 289} (1984), 383--388.
%
\bibitem[GK]{GK} L. Goldberg, L. Keen, \emph{A finiteness theorem for a dynamical class of entire functions}, Ergodic Theory Dynam. Systems {\bf 6} (1986) no.2, 183--192.
%
\bibitem[GM]{GM} L. Goldberg, J. Milnor, \emph{Fixed points of polynomial maps. Part II. Fixed point portraits}, Ann. Sci. \'Ecole Norm. Sup. (4)  {\bf 26}  (1993),  no. 1, 51--98.
%
%\bibitem[GS]{GS} J. Graczyk, G. \'Swiatek \emph{Siegel disks with critical points on their boundaries,} Duke Math. J. {\bf 119} (2003), no.3, 189--196.
%
\bibitem[Ha]{Ha} A. Hatcher, \emph{Algebraic topology}.  Cambridge University Press, Cambridge, 2002.
%
\bibitem[He1]{He1} M.-R. Herman, \emph{Sur la conjugaison diff\'erentiable des diff\'eomorphismes du cercle \'a des rotations}, Inst. Hautes \'Etudes Sci. Publ. Math., (1979), no.49, 5--233 .
%
\bibitem[He2]{He2}  M. R. Herman,
\emph{Are there critical points on the boundaries of singular domains?}
Comm. Math. Phys. {\bf 99} (1985), no. 4, 593-612.
%
\bibitem[He3]{her3} M.~Herman, {\em Conjugaison quasi-sym\'etrique des diff\'ephismes du cercels
et applications aux disques singuliers de Siegel}, manuscript, 1986
%
\bibitem[He4]{He3} M.~Herman, {\em Conjugaison quasi-sym\'etrique des homeomorphismes analytiques
du cercels \`a des rotations}, manuscript 1987.
%
\bibitem[Iv]{Iv} F. Iversen, \emph{Recherches sur les fonctions inverses}, Comptes Rendus {\bf 143} (1906), 877--879; Math Werke, no. 1, 655-656.
%
\bibitem[Ma1]{Ma1} R. Ma\~n\'e, \emph{ Hyperbolicity, sinks and measure in one-dimensional dynamics}, Comm. Math. Phys., {\bf 100} (1985), no. 4 495--524.
%
\bibitem[Ma2]{Ma2} R. Ma\~n\'e, \emph{ Erratum: Hyperbolicity, sinks and measure in one-dimensional dynamics}, Comm. Math. Phys. {\bf 112} ( 1987), no. 4, 721--724.
%
\bibitem[Ma3]{Ma3}R.  Ma\~n\'e, \emph{On a theorem of Fatou}. Bol. Soc. Brasil. Mat. (N.S.) {\bf 24} (1993), no. 1, 1--11. 
%
\bibitem[Mi]{Mi} J. Milnor, \emph{Dynamics in one complex variable}, Annals of Mathematics Studies (2006), Princeton University Press.
%
\bibitem[Re1]{remthesis} L. Rempe, {\em Dynamics of Exponential Maps}, doctoral thesis, Christian-Albrechts-Universit\"at Kiel, 2003, \url{http://http://eldiss.uni-kiel.de/macau/receive/dissertation_diss_00000781}.
%
\bibitem[Re2]{rem04} L. Rempe, {\em On a question of Herman, Baker and Rippon concerning Siegel disks}, Bull.~Lon.~Math.~Soc.~{\bf 36} (2004), 516-518.
%
\bibitem[Re3]{rem06} L. Rempe, \emph{A landing theorem for periodic rays of exponential maps}, 
Proc. Amer. Math. Soc. {\bf 134(9)} (2006), 2639--2648.
%
\bibitem[Re4]{rem08} L. Rempe, \emph{Siegel disks and periodic rays of entire functions}, J. Reine Angew. Math.  {\bf 624} (2008), 81--102.
%

%
\bibitem[RvS]{RvS} L. Rempe, S. Van Strien, \emph{Absence of line fields and Ma\~n\'e's theorem for non-recurrent transcendental functions}.
Trans. Amer. Math. Soc. {\bf 363} (2011), no. 1, 203--228. 
%
\bibitem[R3S]{R3S} G. Rottenfusser, J. R\" uckert, L. Rempe, D. Schleicher,  \emph{Dynamic rays of bounded-type entire functions}, Annals of Mathematics  {\bf 173} (2010), 77--125.
%
\bibitem[Ro]{Ro} J. T. Rogers, \emph{Diophantine conditions imply critical points on the boundaries of Siegel disks of polynomials}, Comm. Math. Phys. {\bf 195} (1998), no. 1, 175--193. 
%
%\bibitem[S]{swi} G.~\'Swi\c atek, {\em Rational Rotation Numbers for Maps of the Circle}, Comm.~Math.~Phys.,
%{\bf 119} (1988), 109--128.
%
%\bibitem[Z1]{zak1} S.~Zakeri, {\em Dynamics of cubic Siegel polynomials}, Comm.~Math.~Phys.,
%{\bf 206} (1999), 185--233.
%%
%\bibitem[Z2]{zak2} S.~Zakeri, {\em On Siegel disks of a class of entire maps}, Duke Math J. {\bf 152} (2010), 481--532.
%%
%\bibitem[Zh]{zha} G. Zhang, {\em All bounded type Siegel disks of rational maps are quasidisks}, Invent.~Math.~{\bf 185} (2011), 421--466.
%
\bibitem[Yo]{Yo} J.-C. Yoccoz, \emph{Analytic linearization of circle diffeomorphisms}. Dynamical systems and small divisors (Cetraro, 1998), volume 1784 of \emph{Lecture notes in Math.}, Springer, Berlin (2002),  125--173.
%
\bibitem[Yn]{Yn} K. Yoneyama, {\em Theory of Continuous Set of Points}, Tohoku Math. J. {\bf 12} (1917),  43--158. 

}
\end{thebibliography}
